\documentclass{article}
\usepackage{geometry}
\usepackage{titling}
\AtEndDocument{\bigskip{\footnotesize%
\textsc{$^1$School of Computing, Dehradun Institute of Technology, Dehradun, 248009, India} \par  
  \textit{E-mail address:} \texttt{suryagiri456@gmail.com} \par

  \addvspace{\medskipamount}
 \textsc{$^*$Department of Applied Mathematics, Delhi Technological University, Delhi–110042, India} \par
 \textit{E-mail address:} \texttt{spkumar@dtu.ac.in} \par
}}
\usepackage{doc}

\usepackage{url}

\usepackage{graphicx}
\usepackage{epstopdf}
\usepackage{amsthm}
\usepackage{amssymb}
\usepackage{amsmath}
\usepackage{enumitem}

\usepackage{array}
\usepackage{enumitem}
\usepackage[utf8]{inputenc}
\usepackage[english]{babel}
\newtheorem{theorem}{Theorem}

\newtheorem{lemma}[theorem]{Lemma}
\newtheorem*{definition}{Definition}
\newtheorem*{remark}{Remark}

\DeclareMathOperator{\RE}{Re}

\usepackage{varwidth}
\usepackage{lineno, hyperref}
\usepackage{graphicx}
\usepackage{epstopdf}
\usepackage{amsmath}
\usepackage{amsthm}
\usepackage{enumitem}
\usepackage{amsthm}
\usepackage{float}
\usepackage{subcaption}
\usepackage{caption}
\usepackage{authblk}
\usepackage{abstract}

\begin{document}
\title{Generalized Toeplitz determinants for Starlike Mappings in Several Complex Variables}
\author{Surya Giri$^{1}$ and S. Sivaprasad Kumar$^*$  }


\date{}


	

\maketitle	

\begin{abstract}
    \noindent This paper establishes sharp bounds for the second and third-order Toeplitz determinants associated with starlike functions $f$ in the unit disk such that $f(z)-z$ has a zero of order $k+1$ at $z=0$. These bounds are further extended to starlike mappings defined on the unit ball in a complex Banach space and on bounded starlike circular domains in $\mathbb{C}^n$. The derived results generalize several known bounds as special cases.
\end{abstract}
\vspace{0.5cm}

\section{Introduction}

    Let $\mathcal{S}$ be the class of analytic univalent  functions $f$ on the open unit disk $\mathbb{U}$ having the normalized form
\begin{equation}
    f(z) =z +\sum_{n=2}^\infty a_n z^n , \quad z \in \mathbb{U}.
\end{equation}
    A well-known subclass of $\mathcal{S}$ is the class of starlike functions, denoted by $\mathcal{S}^*$. A function $f \in \mathcal{S}$ is said to be starlike if and only if
    $$\RE \bigg(\frac{z f'(z)}{f(z)}\bigg) > 0.$$
    Toeplitz matrices and their determinants arise in both pure and applied mathematics. The survey article by Ye and Lim~\cite{YinLim} provides a complete account of their applications in various areas of mathematics. Ali et al.~\cite{AliThoVas} considered the Toeplitz determinants for $f(z)=z+\sum_{n=2}^\infty a_n z^n$, given by
\begin{equation*}
     T_{m,n}(f)= \begin{vmatrix}
	a_n & a_{n+1} & \cdots & a_{n+m-1} \\
	a_{n+1} & a_n & \cdots & a_{n+m-2}\\
	\vdots & \vdots & \ddots & \vdots\\
    a_{n+m-1} & a_{n+m-2} & \cdots & a_n\\
	\end{vmatrix}.
\end{equation*}
   In particular,
   $$  T_{2,2}(f)=  a_2^2 -a_{3}^2  $$
   and
\begin{equation*}\label{T31}
    T_{3,1}(f) =
   \begin{vmatrix}
     1 & a_{2} & a_{3} \\
	 {a}_{2} & 1 & a_{2}\\
	 {a}_3 & {a}_{2} & 1\\
    \end{vmatrix}
   =   1 -  a_3^2 - 2 a_2^2 + a_2^{2} {a}_3 .
\end{equation*}
    Ali et al.~\cite{AliThoVas} derived sharp estimates for $ \vert T_{2,1}(f)\vert$ and $\vert T_{3,1}(f)\vert$ when $f\in \mathcal{S}^*$ and for other subclasses of $\mathcal{S}$. Motivated by this, several authors investigated the same problem for different subclasses  of analytic functions~\cite{AhuKhaRav,GirKum,LecSimBsm}.

     Although estimates for Toeplitz determinants for various subclasses of $\mathcal{S}$ have been established, only a few results are known for biholomorphic mappings in several complex variables. Extending the notion of Toeplitz determinants to higher dimensions, Giri and Kumar~\cite{GirKum3} obtained sharp estimates of second and third-order Toeplitz determinants for a class of biholomorphic mappings on the unit polydisk in $\mathbb{C}^n$ and on the unit ball in a complex Banach space. The same problem was also stuided on bounded starlike circular domains by Giri and Kumar~\cite{GirKum}.

     Let $X$ be a complex Banach space equipped with a norm $\|\cdot \|$ and $\mathbb{B}$ be the unit ball in $X$. The set of all holomorphic mappings from $\mathbb{B}$ into $X$ is denoted by $\mathcal{H}(\mathbb{B})$. If $f\in \mathcal{H}(\mathbb{B})$, then for each integer $k \geq 1$, there exists a bounded symmetric linear mapping
    $$D^k f(z) : \prod_{j=1}^k X \rightarrow X, $$
    called the $k^{th}$ order Fr\'{e}chet derivative of $f$ at $z$. Moreover,  for every $w$ in some neighborhood of $z$, the mapping $f$ can be expressed as
    $$ f(w) = \sum_{k=0}^\infty \frac{1}{k!} D^k f(z) ((w -z)^k ),$$
   where $ D^0 f(z) ((w -z)^0 )= f(z)$
    and for $k \geq 1$,
    $$ D^k f(z)( (w -z)^k) = D^k f(z) \underbrace{( w-z, w-z, \cdots, w-z) }_\text{ k -times}.$$
    On a bounded circular domain $\Omega \subset \mathbb{C}^n$, the first and the $m^{th}$ Fr\'{e}chet derivative of a holomorphic mapping $f : \Omega \rightarrow X$  are written by
    $ D f(z)$ and $D^m f(z) (a^{m-1},\cdot)$, respectively. The matrix representations are
\begin{align*}
    D f(z) &= \bigg(\frac{\partial f_j}{\partial z_k} \bigg)_{1 \leq j, k \leq n}, \\
    D^m f(z)(a^{m-1}, \cdot) &= \bigg( \sum_{p_1,p_2, \cdots, p_{m-1}=1}^n  \frac{ \partial^m f_j (z)}{\partial z_k \partial z_{p_1} \cdots \partial z_{p_{m-1}}} a_{p_1} \cdots a_{p_{m-1}}   \bigg)_{1 \leq j,k \leq n},
\end{align*}
   where $f(z) = (f_1(z), f_2(z), \cdots f_n(z))'$ and $ a= (a_1, a_2, \cdots a_n)'\in \mathbb{C}^n.$
      Let $L(X,Y)$ denote the set of contiuous linear operators from $X$ into a complex Banach space $Y$. For each $z\in X\setminus\{0\}$, let
   $$ T(z)= \left\{ l_z\in L(X,\mathbb{C}): l_z(z) = \| z\|, \| l_z \| = 1 \right\},$$
    where $L(X, Y)$ denotes the set of continuous linear operators from $X$ into a complex Banach space $Y$.  By the Hahn-Banach theorem, the set $T_z$ is non-empty.

   A mapping $f\in \mathcal{H}(\mathbb{B})$ is said to be biholomorphic if $f(\mathbb{B})$ is a domain and the inverse of $f$ exists and is holomorphic on $f(\mathbb{B})$. If Fr\'{e}chet derivate of $f\in \mathcal{H}(\mathbb{B})$  has a bounded inverse for each $z\in \mathbb{B}$, then $f$ is called locally biholomorphic mapping. Analogues to one dimensional case, $f$ is said to be normalized if $f(0)=0$ and $D f(0)=I$, where $I$ represents the linear identity operator from $X$ into $X$.

\begin{definition}\cite{LinHon}
   Suppose $\mathcal{D}$ is a domain in $X$ containing $0$ and $f: \mathcal{D}\rightarrow X$ be holomorphic. We say that $f$ has a zero of order $k$ at $z=0$ if $f(0)=D f(0)=\cdots=D^{k-1}f(0)=0$ and $D^{k}f(0)=0$, where $k\in \mathbb{N}$.
\end{definition}
   Xu~\cite{Xu} investigated the Fekete-Szeg\"{o} problem for starlike mappings $f$ in higher dimensions, where $f(z)-z$ has a zero of order $k+1$ at $z=0$. Giri and Kumar~\cite{GirKum4} determined the Toeplitz determinants for quasi-convex mappings satisfying the same zero-order condition on the unit ball in $X$ and on the unit polydisk in $\mathbb{C}^n$. In this paper, we establish sharp estimates of $\vert T_{2,2}(f)\vert$ and $\vert T_{3,1}(f)\vert$ for starlike mappings defined on the unit ball in $X$ as well as on bounded starlike circular domains in $\mathbb{C}^n$, under the same zero-order condition at $z=0$. These results generalize previously known work and extend certain one-dimensional results to higher dimensions.

  Let $\mathcal{P}$ be the class of analytic functions $p$ on $\mathbb{U}$ such that $p(0)=1$ and $\RE p(z)>0$ for all $z\in \mathbb{U}$.  We use the following lemmas to prove  main results.
\begin{lemma}\cite{Suf}
    Let $f: \mathbb{B} \rightarrow X$ be a normalized locally biholomorphic mapping. The mapping $f$ is said to be starlike on $\mathbb{B}$ if and only if
    $$ \RE (l_z [Df(z)]^{-1}f(z))> 0, \quad x\in \mathbb{B}\setminus \{0\}, \quad l_z \in T_z. $$
    The class of all starlike mappings on $\mathbb{B}$ is denoted by $\mathcal{S}^*(\mathbb{B})$.
\end{lemma}
\begin{lemma}\cite{LiuRen}
  $\Omega \subset \mathbb{C}^n$  is said to be a bounded starlike circular domain if and only if there exists a unique real-valued continuous function
 $\rho: \mathbb{C}^n \rightarrow \mathbb{R},$ called the Minkowski functional of $\Omega$, such that
\begin{enumerate}
    \item $\rho(z) \geq 0$, $z \in \mathbb{C}^n$; \quad $\rho(z) = 0  \Leftrightarrow z=0$,
    \item $\rho(t z) = \vert t\vert , \rho(z)$, $t \in \mathbb{C}$, $z \in \mathbb{C}^n$,
    \item $\Omega = \{ z \in \mathbb{C}^n : \rho(z) < 1 \}$.
\end{enumerate}
   Furthermore, if $\rho(z)$, $z \in \Omega$, belongs to $\mathcal{C}^1$ except on some lower-dimensional manifold $E \subset \mathbb{C}^n$, then $\rho(z)$ satisfies the following properties:
\begin{align}\label{LmEqn}
   & 2 \frac{\partial \rho(z)}{\partial z} z = \rho(z), \quad z \in \mathbb{C}^n \setminus E, \\
    & 2 \frac{\partial \rho(z)}{\partial z} \bigg\vert_{z=z_0} = 1, \quad z_0 \in \partial \Omega \setminus E, \notag\\
    & \frac{\partial \rho(\lambda z)}{\partial z} = \frac{\partial \rho(z)}{\partial z}, \quad \lambda \in (0,\infty),\; z \in \mathbb{C}^n \setminus E, \notag\\
    & \frac{\partial \rho(e^{i \theta} z)}{\partial z} = e^{-i\theta} \frac{\partial \rho(z)}{\partial z}, \quad \theta \in \mathbb{R},\ z \in \mathbb{C}^n \setminus E, \notag
\end{align}
where
   $\frac{\partial \rho(z)}{\partial z} = \left( \frac{\partial \rho(z)}{\partial z_1}, \dots, \frac{\partial \rho(z)}{\partial z_n} \right).$
\end{lemma}

\begin{lemma}\cite{LiuRen}
   Let $\Omega \subset \mathbb{C}^n$ be a bounded starlike circular domain containing $0$, whose Minkowski functional $\rho(z)$ belongs to $C^1$ except on some lower-dimensional manifolds $E \subset \mathbb{C}^n$. Let $f:\Omega \to \mathbb{C}^n$ be a normalized locally biholomorphic mapping. Then $f$ is starlike on $\Omega$ if and only if
  $$\RE \left( \frac{\partial \rho(z)}{\partial z}\,(Df(z))^{-1} f(z) \right) > 0,\quad z \in \Omega \setminus E. $$
  The class of all starlike mappings on $\Omega$ is denoted by $\mathcal{S}^*(\Omega)$.
\end{lemma}

\section{Main Results}
    In this section, we first obtain sharp estimates for $\vert T_{2,2}(f)\vert$ and $\vert T_{3,1}(f)\vert$ for functions $f\in \mathcal{S}^*$ such that $f(z)-z$ has a zero of order $k+1$ at $z=0$. We then extend these results to higher dimensions by deriving the corresponding estimates for the classes $\mathcal{S}^*(\mathbb{B})$ and $\mathcal{S}^*(\Omega)$.
\begin{theorem}\label{thmU}
     If $f(z)=z+\sum_{n=1}^\infty a_{nk+1} z^{nk+1}\in \mathcal{S}^*$, then
     $$ \vert a_{k+1}^2 -  a_{2k+1}^2 \vert \leq \frac{4}{k^2}+\frac{(k+2)^2}{k^4}. $$
     The estimate is sharp.
\end{theorem}
\begin{proof}
     Since $f \in \mathcal{S}^*$, there exists a function $p(z)= 1+ \sum_{n=1}^\infty p_n z^n \in \mathcal{P}$ such that
     $$ z f'(z)=f(z) p(z). $$
     Comparison of same powers of $z$ after the series expansions of functions $f$ and $p$ yields
     $$ a_{k+1} = \frac{p_k}{k} \;\; \text{and}\;\; a_{2k+1} =\frac{k p_{2k}+p_k^2}{2 k^2}. $$
     Using the estimate $\vert p_n \vert \leq 2$ $(n \in \mathbb{N})$ for functions $p\in \mathcal{P}$~\cite[Theorem 3.1.2]{ThoB}, we obtain
\begin{equation}\label{a2Thm1}
      \vert a_{k+1}\vert \leq \frac{2}{k} \;\;\text{and}\;\; \vert a_{2k+1} \vert \leq\frac{k+ 2}{k^2}.
\end{equation}
    In view of these estimate, we deduce that
    $$  \vert a_{k+1}^2 -  a_{2k+1}^2 \vert \leq \vert a_{k+1}\vert^2 + \vert a_{2k+1} \vert^2 \vert \leq \frac{4}{k^2}+\frac{(k+2)^2}{k^4}. $$

     To establish the sharpness, consider the function
\begin{equation}\label{ExtThm1}
     f(z)= \frac{z}{(1-i z^k)^{1/k}}= z+ \frac{2 i}{k} z^{k+1}- \frac{k+2}{k^2} z^{2k+1}+ \cdots.
\end{equation}
    It is evident that $f \in \mathcal{S}^*$ and for the mapping $f$, we have
    $$ \vert a_{k+1}^2 -  a_{2k+1}^2 \vert=  \frac{4}{k^2}+\frac{(k+2)^2}{k^4},$$
    thereby completing the proof.
\end{proof}
\begin{theorem}\label{thmU2}
      If $f(z)=z+\sum_{n=1}^\infty a_{nk+1} z^{nk+1}\in \mathcal{S}^*$, then
\begin{equation*}
\begin{aligned}
    \Big\vert 1 -a_{2 k+1}^2 -2 a_{k+1}^2 +2 a_{k+1}^2 a_{2 k+1}  \Big\vert \leq
\left\{
\begin{array}{ll}
   1+ \dfrac{8}{k^2} + \dfrac{(k+2)(6-k)}{k^4};& 1\leq k \leq 3, \\ \ \\
      1+ \dfrac{8}{k^2} + \dfrac{2+k}{k^3}; &  k\geq 3,
\end{array}
\right.
\end{aligned}
\end{equation*}
     The estimate is sharp $k=1,2,3$.
\end{theorem}
\begin{proof}
    For $\lambda \in \mathbb{C}$ and $f\in \mathcal{S}^*$, Xu~\cite[Theorem 3.1]{Xu} derived that
\begin{equation*}
    \left\vert  a_{2 k+1}- \lambda a_{k+1}^2\right\vert \leq \frac{1}{k} \max\left\{1,  \frac{\vert 2+ k - 4 \lambda\vert}{k} \right\},
\end{equation*}
    which immediately provides
\begin{equation}\label{a3a2Thm2}
    \left\vert  a_{2 k+1}-2 a_{k+1}^2\right\vert \leq \frac{1}{k} \max\left\{1,  \frac{\vert  k - 6 \vert}{k} \right\}.
\end{equation}
     Applying the bounds from (\ref{a2Thm1}) and (\ref{a3a2Thm2}) to the inequality
     $$ \vert 1 -a_{2 k+1}^2 -2 a_{k+1}^2 +2 a_{k+1}^2 a_{2 k+1}  \vert \leq 1+ 2 \vert a_{k+1}\vert^2 + \vert a_{2 k+1}\vert \vert a_{2 k+1}-2 a_{k+1}^2\vert, $$
     we obtain the desired result.

     The equality case is attained for the function $f$ defined by (\ref{ExtThm1}).  For this mapping, we have
     $$ \vert 1 -a_{2 k+1}^2 -2 a_{k+1}^2 +2 a_{k+1}^2 a_{2 k+1}  \vert  = 1+ \frac{8}{k^2} + \frac{(k+2)(6-k)}{k^4}, $$
     which shows that the bounds is sharp for $k=1,2,3$.
\end{proof}  
\begin{remark}
   For $k=1$, Theorem~\ref{thmU2} reduces to~\cite[Theorem 2.1(ii)]{AliThoVas}.
\end{remark}
\begin{theorem}\label{thm3}
    Let  $f \in \mathcal{H}(\mathbb{B}, \mathbb{C})$  and $F(z) =  z f(z) \in \mathcal{S}^*(\mathbb{B})$. If $F(z)-z$ has a zero of order $k+1$ at $z=0$,
   then
\begin{equation*}
   \Big\vert \bigg(\frac{ l_z (D^{2k+1} F(0) (z^{2k+1}))}{(2k+1)! \vert\vert z \vert\vert^{2k+1}}\bigg)^2  - \bigg(  \frac{ l_z (D^{k+1} F(0) (z^{k+1}))}{(k+1)! \|z\|^{k+1}}\bigg)^2  \Big\vert \leq \frac{(k+2)^2}{k^4}+\frac{4}{k^2}.
\end{equation*}
   The bound is sharp.
\end{theorem}
\begin{proof}
     Let $z_0 = \frac{z}{\|z \|}$ for  fix $z\in X\setminus \{ 0 \}$. Let us define $ q : \mathbb{U} \rightarrow \mathbb{C}$ such that
\begin{equation*}
    q(\zeta) = \left\{ \begin{array}{ll}
     \dfrac{\zeta}{ l_z ((D F(\zeta z_0))^{-1} F( \zeta z_0) )}, & \zeta \neq 0, \\ \\
    1, & \zeta =0.
    \end{array}
    \right.
\end{equation*}
   Then $q \in \mathcal{H}(\mathbb{U})$ with $q(0)=1$. Further, we have
\begin{align*}
   q(\zeta) =& \frac{\zeta}{l_z ((D F(\zeta z_0))^{-1} F( \zeta z_0) ) }\\
            = &\frac{\zeta}{l_{z_0} ((D F(\zeta z_0))^{-1} F( \zeta z_0) ) } \\
             =& \frac{\| \zeta z_0 \| }{l_{ \zeta z_0} ((D F(\zeta z_0))^{-1} F( \zeta z_0) ) } , \quad  \zeta \in \mathbb{U}.
\end{align*}
  Since $\RE((D F(z))^{-1} F(z) )>0$, we get
\begin{equation}
   \RE q(\zeta)> 0.
\end{equation}
   By employing the same methodology as in~\cite{PfaSuf}, we obtain
   $$ (D F(z))^{-1}F(z)= z \left( \frac{1}{1+\frac{Df(z)z}{f(z)}}  \right)= \frac{ z f(z)}{f(z)+ D f(z) z},  \quad z\in \mathbb{B},$$
   which further yields
\begin{equation}\label{eqn1}
    \frac{\|z \|}{l_z ((DF(z))^{-1} F(z))}= 1 + \frac{D f(z) z}{f(z)}.
\end{equation}
    In view of (\ref{eqn1}), we obtain
\begin{equation}\label{accr}
    q(\zeta) = \frac{\| \zeta z_0 \| }{l_{ \zeta z_0} ((DF(\zeta z_0))^{-1} F( \zeta z_0) ) }  =  1 + \frac{D f(\zeta z_0)\zeta z_0}{f(\zeta z_0)}.
\end{equation}
   With the Taylor series expansions of $q(\zeta)$ and $f(\zeta z_0)$, the above equation leads to
\begin{align*}
   \bigg (1 + q'(0) &\zeta  + \frac{q''(0)}{2} \zeta^2 + \cdots \bigg)\bigg( 1 + \frac{D^k f(0)(z_0^k)}{k!} \zeta^k +  \frac{D^{k+1} f(0)(z_0^{k+1})}{(k+1)!} \zeta^k + \cdots \bigg)  \\
    &=\bigg( 1 + \frac{D^k f(0)(z_0^k)}{k!} \zeta^k +  \frac{D^{k+1} f(0)(z_0^{k+1})}{(k+1)!} \zeta^k + \cdots \bigg) +\bigg( \frac{D^k f(0)(z_0^k)}{(k-1)!} \zeta^k \\
   &\quad\quad\;\;+  \frac{D^{k+1} f(0)(z_0^{k+1})}{k!} \zeta^{k+1} + \cdots \bigg).
\end{align*}
  On comparing the homogeneous terms of identical degree on both sides, we get
\begin{align*}
    \frac{q^{(k)}(0)}{k!}= \frac{k D^k f(0) (z_0^k)}{k!}, \quad  \frac{q^{(2k)}(0)}{(2k)!}=2k \frac{ D^{2k} f(0) (z_0^{2k})}{(2k)!}- k \bigg( \frac{ D^{k} f(0) (z_0^{k})}{k!}\bigg)^2.
\end{align*}
   That is
\begin{equation}\label{eqthree}
    \frac{q^{(k)}(0)\|z\|^k}{k!}= \frac{k D^k f(0) (z)}{k!}, \; \frac{q^{(2k)}(0)\|z\|^{2k}}{(2k)!}=2k \frac{ D^{2k} f(0) (z^{2k})}{(2k)!}- k \bigg( \frac{ D^{k} f(0) (z^{k})}{k!}\bigg)^2.
\end{equation}
    From the relation $F(z)=z f(z)$, we have
\begin{equation*}
    \frac{D^{k+1} F(0)(z^{k+1})}{(k+1)!}= \frac{D^{k} f(0) (z^{k})}{k!}z, \; \frac{D^{2k+1} F(0)(z^{2k+1})}{(2k+1)!}= \frac{D^{2k} f(0) (z^{2k})}{(2k)!}z,
\end{equation*}
    which immediately yields
\begin{equation}\label{eqfour}
     \frac{l_z (D^{k+1} F(0)(z^{k+1}))}{(k+1)!\|z\|}= \frac{D^{k} f(0) (z^{k})}{k!}, \;\; \frac{l_z (D^{2k+1} F(0)(z^{2k+1}))}{(2k+1)!\|z\|}= \frac{D^{2k} f(0) (z^{2k}) }{(2k)!} .
\end{equation}
   From (\ref{eqthree}) and (\ref{eqfour}), it follows that
   $$ \frac{l_z (D^{k+1} F(0)(z^{k+1}))}{(k+1)!\|z\|^{k+1}}= \frac{1}{k}\frac{q^{(k)}(0)}{k!}.$$
    Applying the bound from~\cite[Theorem 3.1.2]{ThoB} for $q\in \mathcal{P}$, we obtain
\begin{equation}\label{A2B}
   \bigg\vert  \frac{l_z (D^{k+1} F(0)(z^{k+1}))}{(k+1)!\|z\|^{k+1}} \bigg\vert \leq \frac{2}{k}.
\end{equation}
   For $\lambda \in \mathbb{C}$ and $z \in \mathbb{B}\setminus\{0\}$, Xu~\cite{Xu} established that
\begin{equation}\label{CFFSBC}
   \Big\vert  \frac{ l_z (D^{2k+1} F(0) (z^{2k+1}))}{(2k+1)! \vert\vert z \vert\vert^{2k+1}} -  \lambda \Big(\frac{ l_z (D^{k+1} F(0) (z^{k+1}))}{(k+1)! \vert\vert z \vert\vert^{k+1}} \Big)^2 \Big\vert \leq \frac{1}{k} \max \Big\{ 1,  \frac{\vert 2+k - 4 \lambda  \vert}{k } \Big\}.
\end{equation}
   From (\ref{CFFSBC}), we directly deduce that
\begin{equation}\label{3rd}
   \bigg\vert  \frac{ l_z (D^{2k+1} F(0) (z^{2k+1}))}{(2k+1)! \vert\vert z \vert\vert^{2k+1}} \bigg\vert \leq  \frac{ 2+ k   }{k^2 }.
\end{equation}
   Applying the bounds given in (\ref{A2B}) and (\ref{3rd}) to the inequality
\begin{align*}
      \bigg\vert \bigg(\frac{ l_z (D^{2k+1} F(0) (z^{2k+1}))}{(2k+1)! \vert\vert z \vert\vert^{2k+1}}&\bigg)^2 - \bigg(  \frac{ l_z (D^{k+1} F(0) (z^{k+1}))}{(k+1)! \|z\|^{k+1}}\bigg)^2  \bigg\vert \\
      & \leq \bigg\vert \frac{ l_z (D^{2k+1} F(0) (z^{2k+1}))}{(2k+1)! \vert\vert z \vert\vert^{2k+1}}\bigg\vert^2  + \bigg\vert  \frac{ l_z (D^{k+1} F(0) (z^{k+1}))}{(k+1)! \|z\|^{k+1}}  \bigg\vert^2,
\end{align*}
    the stated bound follows immediately.

    The bound is sharp for the mapping
\begin{equation}\label{ExtG}
    F(z)= \frac{z}{(1- i (l_u (z))^k )^{2/k}}, \quad z\in \mathbb{B}, \quad \|u\|=1.
\end{equation}
   From~\cite{LiuLiu}, we have $F\in \mathcal{S}^*(\mathbb{B})$ and $F(z)-z$ has a zero of order $k+1$ at $z=0$. Further, a simple calcultion yields
   $$ \frac{D^{k+1} F(0) (z^{k+1})}{(k+1)! }= \frac{2 i (l_u(z))^k}{k}z \;\;\text{and} \;\; \frac{ D^{2k+1} F(0) (z^{2k+1})}{(2k+1)! }=- \frac{(k+2) (l_u(z))^{2 k}}{k^2}z. $$
   In view of these equations, we have
   $$ \frac{l_z (D^{k+1} F(0) (z^{k+1}))}{(k+1)! }= \frac{2 i (l_u(z))^k}{k}\|z\|$$
    and
    $$\frac{l_z( D^{2k+1} F(0) (z^{2k+1}))}{(2k+1)! }=- \frac{(k+2) (l_u(z))^{2 k}}{k^2}\|z\|, $$
   respectively. By taking $z=ru$, we get
\begin{equation}\label{aaa}
    \frac{l_z (D^{k+1} F(0) (z^{k+1}))}{(k+1)! }=   \frac{2 i}{k} \;\;\text{and} \;\; \frac{l_z( D^{2k+1} F(0) (z^{2k+1}))}{(2k+1)! }= - \frac{(k+2) }{k^2}.
\end{equation}
    Using these values, we obtain
   $$ \bigg\vert \bigg(\frac{ l_z (D^{2k+1} F(0) (z^{2k+1}))}{(2k+1)! \vert\vert z \vert\vert^{2k+1}}\bigg)^2 - \bigg(  \frac{ l_z (D^{k+1} F(0) (z^{k+1}))}{(k+1)! \|z\|^{k+1}}\bigg)^2  \bigg\vert = \frac{(k+2)^2}{k^4}+\frac{4}{k^2}, $$
   which confirms the sharpness of the result.
\end{proof}
\begin{remark}
    When $X=\mathbb{C}$ and $\mathbb{B}=\mathbb{U}$, Theorem~\ref{thm3} is equivalent to Theorem~\ref{thmU}.
\end{remark}
\begin{theorem}\label{thmB1}
   Let  $f \in \mathcal{H}(\mathbb{B}, \mathbb{C})$  and $F(z) =  z f(z) \in \mathcal{S}^*(\mathbb{B})$. If $F(z)-z$ has a zero of order $k+1$ at $z=0$,
   then
\begin{equation*}
\begin{aligned}
    \Big\vert  1 -a_{2 k+1}^2 -2 a_{k+1}^2 +2 a_{k+1}^2 a_{2 k+1}  \Big\vert \leq
\left\{
\begin{array}{ll}
   1+ \dfrac{8}{k^2} + \dfrac{(k+2)(6-k)}{k^4};& 1\leq k \leq 3, \\ \ \\
      1+ \dfrac{8}{k^2} + \dfrac{2+k}{k^3}; &  k\geq 3,
\end{array}
\right.
\end{aligned}
\end{equation*}
    where
\begin{equation}\label{a2a3}
   a_{2k+1} = \frac{ l_z (D^{2k+1} F(0) (z^{2k+1}))}{(2k+1)! \vert\vert z \vert\vert^{2k+1}} \;\;\text{and}\;\; a_{k+1} =  \frac{ l_z (D^{k+1} F(0) (z^{k+1}))}{(k+1)! \|z\|^{k+1}}.
\end{equation}
   The bound is sharp for $k=1,2,3$.
\end{theorem}
\begin{proof}
    From (\ref{CFFSBC}), it follows directly that
\begin{equation}\label{FSS}
   \Big\vert  \frac{ l_z (D^{2k+1} F(0) (z^{2k+1}))}{(2k+1)! \vert\vert z \vert\vert^{2k+1}} -  2 \Big(\frac{ l_z (D^{k+1} F(0) (z^{k+1}))}{(k+1)! \vert\vert z \vert\vert^{k+1}} \Big)^2 \bigg\vert
    \leq \frac{1}{k} \max \Big\{ 1,  \frac{\vert k - 6  \vert}{k } \Big\} .
\end{equation}
   The triangle inequality combined with $(\ref{T31})$ gives
   $$\vert T_{3,1}(F) \vert \leq 1+ 2 \vert a_{k+1}\vert^2+\vert a_{2k+1}\vert \vert a_{2k+1} - 2 a_{k+1}^2 \vert , $$
   where $a_{2k+1}$  and $a_{k+1}$ are given as in (\ref{a2a3}). Incorporating the estimates from (\ref{A2B}), (\ref{3rd}) and ({\ref{FSS}}) into the above inequality yields the required bound.

  When $1\leq k\leq 3$, the bound is sharp for the mapping $F$ defined by (\ref{ExtG}). For this mapping, it follows from (\ref{aaa}) that
   $$ a_{k+1}=   \frac{2 i}{k} \;\;\text{and} \;\; a_{2k+1}= - \frac{(k+2) }{k^2}. $$
   Consequently, a direct calculation yields
   $$   1 -a_{2 k+1}^2 -2 a_{k+1}^2 +2 a_{k+1}^2 a_{2 k+1} = 1+ \frac{8}{k^2} + \frac{(k+2)(6-k)}{k^4}, $$
   thereby establishing sharpness for $k=1,2,3$.
\end{proof}
\begin{remark}
    When $X=\mathbb{C}$ and $\mathbb{B}=\mathbb{U}$, Theorem~\ref{thmB1} reduces to Theorem~\ref{thmU2}.
\end{remark}
\begin{theorem}\label{thm5}
   Let  $f \in \mathcal{H}(\Omega, \mathbb{C})$  and $F(z) =  z f(z) \in \mathcal{S}^*(\Omega)$. If $F(z)-z$ has a zero of order $k+1$ at $z=0$,
   then
\begin{equation*}
   \bigg\vert \bigg(2 \frac{\partial \rho(z)}{\partial z}\frac{D^{2k+1} F(0)(z^{2k+1})}{(2k+1)! \rho^{2k+1}(z)}\bigg)^2 -\bigg( 2 \frac{\partial \rho(z)}{\partial z}\frac{D^{k+1} F(0)(z^{k+1})}{(k+1)! \rho^{k+1}(z)} \bigg)^2\bigg\vert \leq  \frac{(k+2)^2}{k^4} + \frac{4}{k^2}.
\end{equation*}
    The estimate is sharp.
\end{theorem}
\begin{proof}
   For $z\in \Omega\setminus E$, let $z_0= \frac{z}{\rho(z)}$. Consider the function $h : \mathbb{U}\rightarrow \mathbb{C}$ defined by
\begin{equation*}
    h(\zeta) = \left\{ \begin{array}{ll}
     \dfrac{\zeta }{2 \frac{\partial \rho(z_0)}{\partial z} (D F(\zeta z_0))^{-1} F(\zeta z_0)}, & \zeta \neq 0, \\ \\
    1, & \zeta =0.
    \end{array}
    \right.
\end{equation*}
     Then $h \in \mathcal{H}(\mathbb{U})$ and $h(0)=1$. Since $F\in \mathcal{S}^*(\Omega)$, we have
\begin{align*}
    \RE (h(\zeta))&= \RE \bigg(   \frac{\zeta }{2 \frac{\partial \rho(z_0)}{\partial z} (D F(\zeta z_0))^{-1} F(\zeta z_0)}\bigg) \\
                 &= \RE \bigg(   \frac{\rho(\zeta z_0 )}{2 \frac{\partial \rho(\zeta z_0)}{\partial z} (D F(\zeta z_0))^{-1} F(\zeta z_0)}\bigg) > 0, \quad \zeta \in \mathbb{U}.
\end{align*}
   Proceeding as in the proof of Theorem~\ref{thmB1}, we obtain
    $$ (D F(z))^{-1}F(z)= z \left( \frac{1}{1+\frac{Df(z)z}{f(z)}}  \right)= \frac{ z f(z)}{f(z)+ D f(z) z},  \quad z\in \Omega \setminus\{ 0\},$$
   which implies that
\begin{equation*}\label{accr2}
    \frac{\rho(z) }{ 2 \frac{\partial \rho( z)}{\partial z} (DF( z))^{-1} F(  z)  }  =  1 + \frac{D f(z)z}{f(z)}, \quad z\in \Omega\setminus E.
\end{equation*}
   From~(\ref{accr2}), we have
   $$  h(\zeta) =  \frac{\rho (\zeta z_0) }{2 \frac{\partial \rho(\zeta z_0)}{\partial z} (D F(\zeta z_0))^{-1} F(\zeta z_0)} =  1 + \frac{D f( \zeta z_0)\zeta z_0}{f(\zeta z_0)}. $$
   Comparison of the same homogeneous expansions terms expanding in powers of $\zeta$ yields
\begin{equation*}
    \frac{ h^{(k)}(0)}{k!}= \frac{k D^k f(0) (z_0^k)}{k!},
\end{equation*}
    which further gives
\begin{equation}\label{handF}
    \frac{ D^k f(0) (z^k)}{k!}= \frac{1}{k}\frac{ h^{(k)}(0) \rho^k(z)}{k!}.
\end{equation}
   Moreover, from the relation $F(z)=z f(z)$, it follows that
\begin{equation}\label{eqnnew}
     \frac{D^{k+1} F(0)(z^{k+1})}{(k+1)!}= \frac{D^{k} f(0) (z^{k})}{k!} z. 
\end{equation}
    From~(\ref{LmEqn}) and (\ref{eqnnew}), we have
\begin{equation}\label{Fandrho}
    2 \frac{\partial \rho(z)}{\partial z}  \frac{D^{k+1} F(0)(z^{k+1})}{(k+1)!} =  \frac{D^{k} f(0) (z^{k}) \rho(z)}{k!} .
\end{equation}
   Combining (\ref{handF}) with (\ref{Fandrho}), we deduce that
\begin{equation}
      2 \frac{\partial \rho(z)}{\partial z}  \frac{D^{k+1} F(0)(z^{k+1})}{(k+1)! \rho^{k+1}(z)} = \frac{1}{k}\frac{ h^{(k)}(0) }{k!}.
\end{equation}
   Using the bound from~\cite[Theorem 3.1.2]{ThoB} for the Carath\'{e}odory function $h$, we get
\begin{equation}\label{b2}
    \bigg\vert 2 \frac{\partial \rho(z)}{\partial z}  \frac{D^{k+1} F(0)(z^{k+1})}{(k+1)! \rho^{k+1}(z)}\bigg\vert \leq \frac{2}{k}.
\end{equation}
   For $\lambda \in \mathbb{C}$, Xu~\cite[Theorem 3.4]{Xu} proved that
\begin{equation}\label{b33}
\begin{aligned}
\left.
\begin{array}{ll}
    \bigg\vert 2 \dfrac{\partial \rho(z)}{\partial z}\dfrac{D^{2k+1} F(0)(z^{2k+1})}{(2k+1)! \rho^{2k+1}(z)} &- \lambda \bigg( 2 \dfrac{\partial \rho(z)}{\partial z}\dfrac{D^{k+1} F(0)(z^{k+1})}{(k+1)! \rho^{k+1}(z)} \bigg)^2 \bigg\vert\\
  &\leq \dfrac{1}{k} \max\bigg\{ 1, \dfrac{\vert 2+ k - 4 \lambda \vert}{k} \bigg\},  \quad z \in \Omega \setminus E.
\end{array}
\right\}
\end{aligned}
\end{equation}
   As a direct consequence, the above inequality readily yields
\begin{equation}\label{b3}
       \Bigg\vert  2 \frac{\partial \rho(z)}{\partial z}\frac{D^{2k+1} F(0)(z^{2k+1})}{(2k+1)! \rho^{2k+1}(z)} \Bigg\vert \leq \frac{k+2}{k^2}.
\end{equation}
   The bounds given in (\ref{b2}) and (\ref{b3}), together with the inequality
\begin{align*}
    \bigg\vert \bigg( 2 \frac{\partial \rho(z)}{\partial z} &\frac{D^{2k+1} F(0)(z^{2k+1})}{(2k+1)! \rho^{2k+1}(z)} \bigg)^2 -  \bigg( 2 \frac{\partial \rho(z)}{\partial z}\frac{D^{k+1} F(0)(z^{k+1})}{(k+1)! \rho^{k+1}(z)} \bigg)^2 \bigg\vert \\
      &\quad\quad\quad \quad  \leq \bigg\vert  2 \frac{\partial \rho(z)}{\partial z}\frac{D^{2k+1} F(0)(z^{2k+1})}{(2k+1)! \rho^{2k+1}(z)} \bigg\vert^2 +  \bigg\vert 2 \frac{\partial \rho(z)}{\partial z}\frac{D^{k+1} F(0)(z^{k+1})}{(k+1)! \rho^{k+1}(z)} \bigg\vert^2,
\end{align*}
   lead to the required  result.

    To complete the sharpness part, we consider the mapping
\begin{equation}\label{ExtOmega}
   F(z)= \frac{z}{\Big(1 - i \left(\dfrac{z_1}{r} \right)^k \Big)^{\frac{2}{k}}}, \quad z\in \Omega,
\end{equation}
   where $r=\sup\{\vert z_1\vert : z=(z_1,0,\cdots,0)'\in \Omega \}$. According to~\cite{LiuLiu2}, the mapping given by~(\ref{ExtOmega}) belongs to $\mathcal{S}^*(\Omega)$ and $z=0$ is the zero of order $k+1$ of $F(z)-z$. For this mapping, we have
    $$  \frac{D^{k+1}  F(0) (z^{k+1})}{(k+1)!}= i\Big(\frac{2}{k}\Big)  \Big(\frac{z_1}{r}\Big)^k z \;\; \text{and} \;\; \frac{D^{2k+1}  F(0) (z^{2k+1})}{(2k+1)!} = - \Big(\frac{k+2}{k^2}\Big) \Big(\frac{z_1}{r}\Big)^{2k} z. $$
    Applying (\ref{LmEqn}) in the above equations, we get
\begin{align*}
    2 \frac{\partial \rho}{\partial z}  \frac{D^{k+1}  F(0) (z^{k+1})}{(k+1)! } = i\Big(\frac{2}{k}\Big)  \Big(\frac{z_1}{r}\Big)^k \rho (z)
\end{align*}
    and
\begin{align*}
      2 \frac{\partial \rho}{\partial z} \frac{D^{2k+1}  F(0) (z^{2k+1})}{(2k+1)!} = - \Big(\frac{k+2}{k^2}\Big) \Big(\frac{z_1}{r}\Big)^{2k} \rho(z),
\end{align*}
   respectively.  Taking $z = R u$ $(0< R <1)$, where $u =(u_1, u_2, \cdots, u_n)' \in \partial \Omega$ and $u_1 =r$, we obtain
\begin{align}\label{b2H3}
     2 \frac{\partial \rho}{\partial z}  \frac{D^{k+1}  F(0) (z^{k+1})}{(k+1)! \rho^{k+1} (z)} = i\Big(\frac{2}{k}\Big)
\end{align}
   and
\begin{align}\label{b3H3}
        2 \frac{\partial \rho}{\partial z} \frac{D^{2k+1}  F(0) (z^{2k+1})}{(2k+1)! \rho^{2k+1} (z)} = - \Big(\frac{k+2}{k^2}\Big) .
\end{align}
   Consequently, in view of (\ref{b2H3}) and (\ref{b3H3}), we have
\begin{align*}
      \bigg\vert   \bigg(  2 \frac{\partial \rho}{\partial z} \frac{D^{2k+1}  F(0) (z^{2k+1})}{(2k+1)! \rho^{2k+1} (z)} \bigg)^2 -  \bigg(2 \frac{\partial \rho}{\partial z}  \frac{D^{k+1}  F(0) (z^{k+1})}{(k+1)! \rho^{k+1} (z)} \bigg)^2  \bigg\vert    =   \frac{(k+2)^2}{k^4}+ \frac{4}{k^2},
\end{align*}
    which establishes the sharpness of the bound.
\end{proof}
\begin{remark}
    When $n=1$ and $\Omega=\mathbb{U}$, Theorem~\ref{thm5} is equivalent to Theorem~\ref{thmU}.
\end{remark}
\begin{theorem}\label{thm6}
Let  $f \in \mathcal{H}(\Omega, \mathbb{C})$  and $F(z) =  z f(z) \in \mathcal{S}^*(\Omega)$. If $F(z)-z$ has a zero of order $k+1$ at $z=0$,
   then
\begin{equation*}
\begin{aligned}
    \Big\vert  1 -a_{2 k+1}^2 -2 a_{k+1}^2 +2 a_{k+1}^2 a_{2 k+1}  \Big\vert \leq
\left\{
\begin{array}{ll}
   1+ \dfrac{8}{k^2} + \dfrac{(k+2)(6-k)}{k^4};& 1\leq k \leq 3, \\ \ \\
      1+ \dfrac{8}{k^2} + \dfrac{2+k}{k^3}; &  k\geq 3,
\end{array}
\right.
\end{aligned}
\end{equation*}
    where
\begin{equation}\label{b2b3}
   a_{2k+1} =  2 \frac{\partial \rho(z)}{\partial z}\frac{D^{2k+1} F(0)(z^{2k+1})}{(2k+1)! \rho^{2k+1}(z)} \;\;\text{and}\;\; a_{k+1} = 2 \frac{\partial \rho(z)}{\partial z}\frac{D^{k+1} F(0)(z^{k+1})}{(k+1)! \rho^{k+1}(z)}.
\end{equation}
   The bound is sharp for $k=1,2,3$.
\end{theorem}
\begin{proof}   In view of (\ref{b33}) and (\ref{b2b3}), we immediately get
\begin{equation}\label{b3b22}
    \Big\vert a_{2k+1}- 2 a_{k+1}^2  \Big\vert\leq \frac{1}{k} \max\Big\{ 1, \frac{\vert k - 6 \vert}{k} \Big\}.
\end{equation}
    Substituting the bounds from (\ref{b2}), (\ref{b3}) and (\ref{b3b22}) into the inequality
\begin{align*}
  \Big\vert  1 -a_{2 k+1}^2 -2 a_{k+1}^2 +2 a_{k+1}^2 a_{2 k+1}   \Big\vert  \leq 1  +2 \left\vert  a_{k+1} \right\vert^2  + \left\vert a_{2 k+1}  \right\vert \left\vert a_{2 k+1}-2 a_{k+1}^2 \right\vert,
\end{align*}
   gives the asserted bound.

   From (\ref{b2H3}) and (\ref{b3H3}), for the mapping $F$ defined by (\ref{ExtOmega}), it is evident that
   $$  \left\vert   1 -a_{2 k+1}^2 -2 a_{k+1}^2 +2 a_{k+1}^2 a_{2 k+1} \right\vert = 1+ \frac{8}{k^2} + \frac{(k+2)(6-k)}{k^4}. $$
   Hence, the bound is sharp for $1\leq k \leq 3.$
\end{proof}
\begin{remark}
    When $n=1$ and $\Omega=\mathbb{U}$, Theorem~\ref{thm6} is equivalent to Theorem~\ref{thmU2}.
\end{remark}
\section*{Declarations}
\subsection*{Conflict of interest}
	The authors declare that they have no conflict of interest.
\subsection*{Author Contribution}
    Each author contributed equally to the research and preparation of the manuscript.
\subsection*{Data Availability} Not Applicable.
\noindent

\end{document}